\documentclass[11pt]{amsart}
\usepackage[a4paper]{geometry}
\usepackage{amsmath,amsfonts,amssymb}
\usepackage{hyperref}
\def\C{\mathbb C}
\def\R{\mathbb R}
\newtheorem{thm}{Theorem}[section]
\newtheorem{lem}[thm]{Lemma}
\newtheorem{cor}[thm]{Corollary}
\newtheorem{prop}[thm]{Proposition}
\numberwithin{equation}{section}
\theoremstyle{remark}
\newtheorem{rem}{Remark}[section]
\theoremstyle{definition}
\newtheorem{df}{Definition}[section]
\DeclareMathOperator{\diag}{diag}
\DeclareMathOperator{\bdiag}{b-diag}
\DeclareMathOperator{\spec}{spec}

\begin{document}
\author{Kay Jachmann}
\address{Institute of Applied Analysis, TU Bergakademie Freiberg, 09596 Freiberg, Germany}
\email{jachmann@math.tu-freiberg.de}
\author{Jens Wirth}
\address{Department of Mathematics, Imperial College London, 180 Queen's Gate, London SW7 2AZ, UK}
\thanks{The second author is supported by EPSRC with grant EP/E062873/1}
\email{j.wirth@imperial.ac.uk}
\title{Diagonalisation schemes and applications}
\begin{abstract}
These notes develop aspects of perturbation theory of matrices related to so-called diagonalisation schemes. Primary focus is on constructive tools to derive asymptotic expansions for small/large parameters of eigenvalues and eigenprojections of families of matrices depending upon real/complex parameters.

Applications of the schemes in different frameworks are also discussed and references to further applications given. 
\end{abstract}
\subjclass[2000]{Primary 47A56; Secondary 34E05, 35M10, 65F15, 80A17}
\maketitle

\tableofcontents

\section{Introduction}
In various situations one is confronted with the question of determining eigenvalues and eigenprojections of matrices uniform in certain parameters. The available tools for answering this question range from hand calculation (which is fine if matrices are small, i.e., have size $2\times 2$) to strong abstract methods of perturbation theory providing analytic formulas (but no concrete answer in most cases). 
In this note we want to collect results on a particularly useful constructive scheme for calculating
asymptotic expansions for diagonalisers and therefore also eigenvalues and eigenprojections.  
The method is flexible enough to generalise to applications to systems of differential equations or matrices of pseudo-differential operators. 

The approaches presented here generalise those used by Taylor \cite{Taylor:1975}, Yagdjian \cite{Yagdjian:1997}, Wang  \cite{Wang:2003}, \cite{Wang:2003a} and others, including the authors \cite{Wirth:2007b}, \cite{Jachmann:2008}.

The paper is organised as follows. First we will recall in Section~\ref{sec2.1} the well-known case of matrices with distinct eigenvalues and their small perturbations. In Section~\ref{sec2.2} we will present the general multi-step diagonalisation scheme, and its relations to uniformly diagonable matrices will be discussed in Section~\ref{sec2.3}.  Section~\ref{sec3} is devoted to applications of the schemes in different situations.

\section{Diagonalisation schemes}\label{sec2}
\subsection{Non-degenerate matrix families and the standard scheme}\label{sec2.1} We start our presentation by recollecting some well-known aspects from perturbation theory in combination with a merely classical proof using a diagonalisation technique. Let for this $A:\R\to\C^{m\times m}$ be a continuous matrix-valued function depending upon a real (or for some applications also complex) parameter $\rho$ and assume that as $\rho\to0$ the family of matrices has a full asymptotic expansion
\begin{equation}\label{eq:1}
   A(\rho) \sim A_0 + \rho A_1 + \rho^2 A_2 + \cdots,\qquad\rho\to0,
\end{equation}
i.e.,
\begin{equation}
   A(\rho) = A_0 + \rho A_1 + \rho^2 A_2 + \cdots + \rho^N A_N + \mathcal O(\rho^{N+1}),\qquad \forall N,
\end{equation}
for certain (uniquely determined) matrices $A_i\in\C^{m\times m}$. In the case of complex parameters this is just slightly weaker than assuming analyticity of $A$ near $\rho=0$; in the case of real parameters the assumption follows from smoothness of $A$ in $\rho=0$.

Now one might ask how eigenvalues and eigenprojections of the matrices $A(\rho)$ depend on $\rho$
as $\rho\to0$. It is well-known that continuity of $A(\rho)$ implies that eigenvalues depend continuously on $\rho$, but in order to conclude more we have to make assumptions on $A(\rho)$. It is reasonable to define a non-degenerate matrix family as one with distinct eigenvalues.

\begin{df}\label{df21}
We call $A(\rho)$ {\em non-degenerate} in $\rho=0$ if $A_0$ has $m$ distinct eigenvalues.
\end{df}

\begin{thm}\label{thm21}
Assume that $A(\rho)$ is non-degenerate in $\rho=0$. Then there exist uniformly bounded families of invertible matrices $M(\rho)$ with uniformly bounded inverse having full asymptotic expansions as 
$\rho\to0$ and satisfying
\begin{equation}\label{eq:2.2}
   A(\rho) M(\rho) - M(\rho)\Lambda(\rho) = \mathcal O(\rho^N),\qquad \forall N,
\end{equation}
for a diagonal matrix $\Lambda(\rho)$.
\end{thm}

The diagonal matrix $\Lambda(\rho)$ coincides to arbitrary order with the (continuous) diagonal matrix containing the eigenvalues of $A(\rho)$. It has an asymptotic expansion and the corresponding coefficients are determined uniquely. Although the above statement looks rather weak, it is of particular interest for us. We will show how to construct the diagonaliser $M(\rho)$ using a recursion scheme and how the aforementioned uniform bounds and the asymptotic expansion of $\Lambda(\rho)$ arise naturally within the construction. 

The construction forms the core of the more involved multi-step scheme introduced later in Section~\ref{sec2.2} and is the key idea in the diagonalisation-based approaches of Yagdjian, Reissig  and co-authors, see e.g., \cite{Yagdjian:1997}, \cite{Reissig:2000}, \cite{Reissig:2005}, \cite{Kubo:2003}. For more detailed discussions on applications we refer to Section~\ref{sec3}.

\begin{proof}(of Theorem~\ref{thm21}) It is enough to prove that for any number $N$ there exists an interval $I_N=(-\epsilon_N,\epsilon_N)$, a uniformly bounded and invertible matrix function $M_N(\rho)$, $\rho\in I_N$, and a diagonal matrix $\Lambda_N(\rho)$, $\rho\in I_N$, such that 
\begin{equation}\label{eq:23}
   A(\rho)M_N(\rho)-M_N(\rho)\Lambda_N(\rho) = \mathcal O(\rho^{N+1})
\end{equation}
and $M_{N+1}(\rho)-M_N(\rho)=\mathcal O(\rho^{N+1})$. Then any two functions $M(\rho)$, $\Lambda(\rho)$ coinciding with all the $M_N(\rho)$, $\Lambda_N(\rho)$ up to the corresponding orders satisfy the above theorem. But this just means we need to construct functions subject to asymptotic expansions, which is a standard argument of asymptotic analysis.

We are going to construct $M_N(\rho)$ and $\Lambda_N(\rho)$ recursively, and so assume that
they can be written as
\begin{subequations}
\begin{align}
    M_N(\rho) &= M_0(I +\rho M^{(1)}+\rho^2M^{(2)}+\cdots +\rho^N M^{(N)} ),\\
    \Lambda_N(\rho)&=\Lambda_0+\rho\Lambda^{(1)}+\cdots+\rho^{N}\Lambda^{(N)}
\end{align}
\end{subequations}
with certain coefficients $M_0$, $\Lambda_0$, $M^{(j)}$ and $\Lambda^{(j)}$. It is clear by \eqref{eq:23}
with $N=0$ that $A_0M_0=M_0\Lambda_0$ and so $M_0$ must be the diagonaliser of the main part $A_0$ and $\Lambda_0$ the corresponding diagonal matrix consisting of all eigenvalues of $
A_0$.

{\sl Step 1.} We determine a diagonaliser $M_0$ of $A_0$ and set $\tilde A(\rho)=M_0^{-1}A(\rho)M_0$.
The new matrix family $\tilde A(\rho)$ has the asymptotic expansion $\tilde A(\rho)\sim \Lambda_0+\rho\tilde A_1+\rho^2\tilde A_2+\cdots$ as $\rho\to0$ with coefficients $\tilde A_k=M_0^{-1}A_kM_0$.  

{\sl Step 2.} Assume we have already determined $M_{k-1}(\rho)$ and $\Lambda_{k-1}(\rho)$ for a certain number $k=1,2,\ldots$. Then we denote 
\begin{equation}\label{eq:2.5}
   B^{(k)}(\rho) =  A(\rho) M_{k-1}(\rho)-M_{k-1}(\rho) \Lambda_{k-1}(\rho)  = \mathcal O(\rho^{k}).
\end{equation}
It is evident that this matrix family has a full asymptotic expansion as $\rho\to0$, and it makes sense to 
set $\tilde B^{(k)} = \lim_{\rho\to0} \rho^{-k} M_0^{-1} B^{(k)}(\rho)$. We now construct the next coefficient matrices. For this we set $\Lambda^{(k)} = \diag \tilde B^{(k)}$, and take $M^{(k)}$ to be the solution of the Sylvester equation
\begin{equation}\label{eq:2.6}
   [\Lambda_0, M^{(k)} ] + \tilde B^{(k)}-\diag \tilde B^{(k)} = 0
\end{equation}
with vanishing diagonal entries. Equation~\eqref{eq:2.6} is soluble for any $\tilde B^{(k)}$, because $\Lambda_0=\diag(\lambda_{0,1},\ldots,\lambda_{0,n})$ has distinct entries, and its solution is given explicitly by
\begin{equation}\label{eq:2.7}
   \big(M^{(k)}\big)_{ij} = \begin{cases} -\frac{\big(\tilde B^{(k)}\big)_{ij}}{\lambda_{0,i}-\lambda_{0,j}}, \qquad & i\ne j,\\0, \qquad & i=j .\end{cases}
\end{equation}
That these are indeed the right choices for the matrices follows from
\begin{align}\label{eq:2.9}
   B^{(k+1)}(\rho) &= B^{(k)}(\rho) + \rho^{k} \left(A(\rho) M_0M^{(k)}- M_0 M^{(k)}\Lambda_k(\rho)-M_{k-1}(\rho)\Lambda^{(k)}\right)\notag\\
   &=\rho^k M_0 \left( \tilde B^{(k)} +[\Lambda_0,M^{(k)}] - \Lambda^{(k)} \right) + \mathcal O(\rho^{k+1})
   = \mathcal O(\rho^{k+1}),
\end{align}
by \eqref{eq:2.6} and so the construction can proceed recursively. 

It remains to check invertibility of $M_N(\rho)$ on a suitable $I_N$. The matrix $M_0$ is invertible; the second factor of the form $I+\mathcal O(\rho)$ and tends to the identity as $\rho\to0$. Since the group of invertible matrices is open in $\C^{m\times m}$, this implies invertibility of $M_N(\rho)$ for sufficiently small $I_N$.  
\end{proof}

\begin{rem}
Determining the diagonaliser $M_0$ in {\sl Step 1} of the scheme is the only `painful' step of the approach and requires considerable effort. Everything else is explicit and consists of operations acting on entries of matrices (in \eqref{eq:2.7}) or matrix multiplications / additions (in \eqref{eq:2.5}). 
\end{rem}
\begin{rem}
The intervals $I_N$ will in general shrink as $N$ tends to infinity. This is different if $A(\rho)$ is analytic in $\rho$ near $0$ and the corresponding asymptotic series converge uniformly on a small interval (see e.g. \cite[Chapter II]{Kato:1980}). In this case \eqref{eq:2.2} implies that the error term is actually 0.
\end{rem}

For completeness we also mention the following spectral bound, which can be used to estimate the eigenvalues of $A(\rho)$ in finitely many steps of the diagonalisation scheme. 

\begin{prop}\label{prop:22}
Representation \eqref{eq:2.5} implies 
\begin{equation}\label{eq:29}
  \mathrm{dist}\big( \spec A(\rho), \spec \Lambda_{k-1} (\rho) \big) \le \|M_{k-1}^{-1}(\rho)B^{(k)}(\rho)\|
  =  \mathcal O(\rho^{k}).
\end{equation}
\end{prop}
\begin{proof}
The diagonal matrix $\Lambda_{k-1}(\rho)$ is normal, and so its resolvent satisfies the bound
\begin{equation}
  \|(\zeta-\Lambda_{k-1}(\rho))^{-1} \| \le \frac1{ \mathrm{dist} \big(\zeta,\spec \Lambda_{k-1}(\rho) \big)}.
\end{equation}  
Thus for any $\zeta\in\C$ with $\mathrm{dist}(\zeta,\spec \Lambda_{k-1}(\rho))> \|M_{k-1}^{-1}(\rho)B^{(k)}(\rho)\|$
the right-hand side of the resolvent identity
\begin{multline}
  \big (\zeta-\Lambda_{k-1}(\rho) \big)^{-1} \big(\zeta - M_{k-1}^{-1}(\rho)A(\rho)M_{k-1}(\rho) \big) \\= 
   I -\big (  \zeta-\Lambda_{k-1}(\rho) \big)^{-1} M_{k-1}^{-1}(\rho)B^{(k)}(\rho)
\end{multline}
is invertible. But this implies invertibility of the left-hand side and so that $\zeta$ belongs to the resolvent set of $A(\rho)$. 
\end{proof}

\begin{rem}
The bound is almost optimal. If $\rho$ is sufficiently small the right-hand side of the estimate involves essentially all entries of $\tilde B^{(k)}$, while the next step of the diagonalisation brings in the diagonal entries of $\tilde B^{(k)}$ as new coefficients. So the estimate is sharp if $\tilde B^{(k)}$ has no exceptionally large off-diagonal entries.
\end{rem}

Another interesting consequence is that the constructed matrices $M_N(\rho)$ allow to approximate the eigenprojections. To make this precise we consider the eigenvalue $\lambda_j(\rho)$ with its eigenprojection $P_{\lambda_j(\rho)}$. It corresponds to the $j$-th diagonal entry of the matrix $\Lambda_N(\rho)$. For the diagonal matrix it is evident that the corresponding eigenprojection is just $e_j\otimes e_j$ for the $j$-th basis vector $e_j$ of the standard basis. 
The eigenprojection $P_{\lambda_j(\rho)}$ is determined by $  \big(\lambda_j(\rho)-A(\rho)\big) P_{\lambda_j(\rho)} =0 $ (and the fact that it is a projection, i.e.,  $P_{\lambda_j(\rho)} ^2= P_{\lambda_j(\rho)}$).

\begin{prop}\label{prop:23}
 The projection $M_N(\rho)   (e_j\otimes e_j) M_N^{-1}(\rho)$ approximates the eigenprojection $P_{\lambda_j(\rho)}$, in the sense that 
\begin{equation}
  \big(\lambda_j(\rho)-A(\rho)\big) M_N(\rho)   (e_j\otimes e_j) M_N^{-1}(\rho)= \mathcal O(\rho^{N+1})
\end{equation}
and
\begin{equation}
   \| P_{\lambda_j(\rho)} - M_N(\rho)   (e_j\otimes e_j) M_N^{-1}(\rho) \| = \mathcal O(\rho^{N+1}).
\end{equation}
\end{prop}
\begin{proof}
The first formula is a straightforward consequence of \eqref{eq:23} and Proposition~\ref{prop:22}. For the second one we recall that the eigenprojections of a matrix can be represented in terms of the eigenvalues by the product formula
\begin{equation}
   P_{\lambda_j(\rho)} = \prod_{i\ne j} (\lambda_i(\rho)-\lambda_j(\rho))^{-1} (\lambda_i(\rho)-A(\rho)). 
\end{equation}
Plugging in \eqref{eq:23} and \eqref{eq:29}  for $A(\rho)$ and $\lambda_j(\rho)$, respectively, the second statement follows. In both cases the constants in the estimates can be calculated explicitly.
\end{proof} 

\subsection{Block-diagonalisation}\label{sec2.2} The main objective of this note is to discuss how to generalise the scheme from the proof of Theorem~\ref{thm21} to degenerate matrix functions and to replace the assumption of non-degeneracy by weaker assumptions, which are just enough to ensure the existence of an asymptotic expansion of eigenvalues and eigenprojections in {\em entire powers} of $\rho$. 

The main motivation for the approach goes back to Taylor \cite{Taylor:1975}  and Wang \cite{Wang:2003}, \cite{Wang:2003a} and applications to the decoupling of hyperbolic-parabolic coupled systems. The precise construction is taken from the PhD thesis of the first author, \cite{Jachmann:2008}. 

A short comment on notation: for matrices appearing recursively in the scheme we use an upper index $[j]$ to denote the level. Matrices which are final results of the consideration get lower indices according to their position and if these matrices are themselves families with asymptotic expansions we use upper indices $(j)$ to denote their $j$-th term. 

\subsubsection{Basic setting and diagonalisation modulo $\mathcal O(\rho^2)$} 
We assume that the matrix $A_0$ is diagonable, and denote by $M_0$ a diagonaliser of $A_0$, which arranges the eigenvalues in groups, i.e., 
\begin{equation}
   M_0^{-1} A_0 M_0 = \Lambda_0 = \diag(\lambda_1^{[0]},\cdots, \lambda_m^{[0]})
\end{equation}
and $\lambda_i^{[0]}=\lambda_j^{[0]}$, $i< j$,  implies $\lambda_i^{[0]}=\lambda_k^{[0]}$ for all $i\le k\le j$. We apply $M_0$ to the family $A(\rho)$:
\begin{equation}
   \tilde A^{[0]}(\rho) = M_0^{-1} A(\rho) M_0 \sim \Lambda_0 + \rho \tilde A_1^{[0]} + \rho^2\tilde A_2^{[0]}+\cdots,\qquad\rho\to0.
\end{equation}
We assume that $\Lambda_0$ is not just a multiple of the identity and try to follow the standard scheme as far as possible. If we consider the Sylvester equation $[\Lambda_0, X] = Y$,
we see that this can only be solved if $(Y)_{ij}=0$ whenever $\lambda_i=\lambda_j$. More formally, we introduce the notation $\Pi_0=(\pi_1,\ldots,\pi_k)$, $\pi_1+\cdots+\pi_k=m$, for the partition of eigenvalues, and write $i\sim_{\Pi_0}j$ if $\lambda_i^{[0]}=\lambda_j^{[0]}$. Furthermore, $\bdiag_{\Pi_0}$ selects the corresponding diagonal blocks out of a matrix / constructs a block-diagonal matrix of this form.  

Let $\tilde\Lambda_1=\bdiag_{\Pi_0} \tilde A_1^{[0]}$. Then we can solve the Sylvester equation 
\begin{equation}
[\Lambda_0,K^{[0]}_0]+\tilde A_1^{[0]} - \tilde\Lambda_1=0,
\end{equation}
and obtain as possible solution (cf. \eqref{eq:2.7})
\begin{equation}
   \big(K^{[0]}_0\big)_{ij} = \begin{cases} -\frac{\big(\tilde A_1^{[0]}\big)_{ij}}{\lambda_{i}^{[0]}-\lambda_{j}^{[0]}}, \qquad & i\not\sim_{\Pi_0} j,\\0, \qquad & i\sim_{\Pi_0}j .\end{cases}
\end{equation}
Furthermore, 
\begin{equation}
 \tilde A^{[0]}(\rho) (I+\rho K^{[0]}_0) -   (I+\rho K^{[0]}_0) (\Lambda_0+\rho\tilde\Lambda_1) 
 = \rho(  \tilde A_1^{[0]}+[\Lambda_0,K^{[0]}_0]  - \tilde\Lambda_1 ) + \mathcal O(\rho^2)
\end{equation}
is of order $\rho^2$. In order to get the matrix diagonalised modulo $\mathcal O(\rho^2)$ it remains to diagonalise $\tilde \Lambda_1$. For this we observe the following: if $\tilde\Lambda_1$ is diagonable, its diagonaliser lives in the subspaces corresponding to the blocks / the partition $\Pi_0$. On all these subspaces the matrix $\Lambda_0$ is a multiple of the identity (by definition of the partition). This implies that $\Lambda_0$ is invariant under all diagonalisers of $\tilde\Lambda_1$, provided they exist.

Thus we assume that $\tilde\Lambda_1$ is diagonable, and denote by $\tilde M_1$ a diagonaliser of $\tilde\Lambda_1$, arranging the eigenvalues into groups within the the partition $\Pi_0$. Then
\begin{equation}
  M_0 (I+ \rho K^{[0]}_0)\tilde M_1   
\end{equation}
diagonalises $A(\rho)$ modulo $\mathcal O(\rho^2)$ for small values of $\rho$. Hence we constructed\begin{align}
  \tilde A^{[1]}(\rho) &=   
\tilde M_1^{-1}(I+ \rho K^{[0]}_0)^{-1}\tilde A^{[0]}(\rho) (I+ \rho K^{[0]}_0)\tilde M_1 \notag\\
  & \sim \Lambda_0 + \rho \Lambda_1 + \rho^2\tilde A_2^{[1]}+\cdots,\qquad\rho\to0.
\end{align}
Assumptions we had to make were the diagonability of the two matrices $A_0$ and  $\tilde\Lambda_1=\bdiag_{\Pi_0} M_0^{-1}A_1M_0$. 
 
\begin{rem}\label{rem2.3}
Instead of diagonalising $\tilde\Lambda_1$ and stopping the procedure we could also apply the construction iteratively  and obtain a block-diagonaliser 
\begin{equation}
M_1(\rho) \sim I+\rho M^{(1)}+\rho^2 M^{(2)} +\cdots, 
\end{equation}
(with $M^{(1)}=K_0^{[0]}$ and the further terms obtained by a similar procedure to the standard scheme)
such that $M_1^{-1}(\rho) A(\rho)M_1(\rho)$ is $\Pi_0$-block-diagonal modulo $\bigcap_N \mathcal O(\rho^N)$.  
\end{rem} 

\subsubsection{The iterative scheme} We assume we already applied $k$ steps to diagonalise the given family modulo $\mathcal O(\rho^{k+1})$, i.e.,  we assume we are given 
\begin{equation}
  \tilde A^{[k]}(\rho) \sim \Lambda_0 + \rho \Lambda_1 + \cdots + \rho^k \Lambda_k +
  \rho^{k+1} \tilde A^{[k]}_{k+1}+\cdots,\quad \rho\to0,
\end{equation}
and in particular that the eigenvalues of $\Lambda_k$ are arranged into groups within the partition $\Pi_{k-1}$.
Associated to $\Lambda_k$ we have a new partition $\Pi_k$, which is a sub-partition of $\Pi_{k-1}$,
and a corresponding equivalence relation $\sim_{\Pi_k}$. Denoting the entries of $\Lambda_{k'}$   
as $\lambda^{[k']}_1$, \dots, $\lambda^{[k']}_m$ this means
\begin{equation}
  i\sim_{\Pi_k} j\qquad \Longleftrightarrow \qquad \forall k'\le k\;:\; \lambda^{[k']}_i=\lambda^{[k']}_j.
\end{equation}
Our strategy now is as follows. The matrix $\tilde A^{[k]}_{k+1}$ is a full matrix, which does not obey the special block structure determined by the filtration $\Pi_k$ of partitions, but we can construct step by step diagonaliser which eliminate off-diagonal terms modulo $\mathcal O(\rho^{k+2})$ with respect to the $\Pi_{k'}$, $k'\le k$. This can be done in $k+1$ sub-steps. 

For the $0$-th step let $A^{[k,0]}_{k+1} = \bdiag_{\Pi_0}  \tilde A^{[k]}_{k+1}$ and $K^{[k]}_0$ be the solution to the Sylvester equation
\begin{equation}
   [\Lambda_0, K^{[k]}_0] + \tilde A^{[k]}_{k+1} - A^{[k,0]}_{k+1}= 0.
\end{equation}
Then by construction
\begin{equation}
  \tilde A^{[k]}(\rho) (I+\rho^{k+1} K^{[k]}_0) - (I+\rho^{k+1} K^{[k]}_0) (\sum_{j=0}^{k}\rho^j\Lambda_j +\rho^{k+1} A^{[k,0]}_{k+1})
  \end{equation}
is of order $\mathcal O(\rho^{k+2})$, the last term on the right of $\Pi_0$-block-diagonal structure. 

In the $\ell$-th sub-step, $\ell=1,\ldots,k$, we denote $ A^{[k,\ell]}_{k+1} = \bdiag_{\Pi_\ell}  A^{[k,\ell-1]}_{k+1}$, and define $K^{[k]}_\ell$ to be a $\Pi_{\ell-1}$-block-diagonal solution of
\begin{equation}
   [\Lambda_\ell, K^{[k]}_\ell] + A^{[k,\ell-1]}_{k+1} - A^{[k,\ell]}_{k+1}=0.
\end{equation}
Again by construction and the commutation property $[\Lambda_{\ell'}, K^{[k]}_\ell]=0$ for
$\ell'<\ell$ (because $K^{[k]}_\ell$ acts only on the invariant subspaces related to $\Pi_{\ell-1}
$),  it follows that
\begin{equation}
  (\sum_{j=0}^{k}\rho^j\Lambda_j +\rho^{k+1} A^{[k,\ell-1]}_{k+1}) (I+\rho^{k+1-\ell} K^{[k]}_\ell) - (I+\rho^{k+1-\ell} K^{[k]}_\ell) (\sum_{j=0}^{k}\rho^j\Lambda_j +\rho^{k+1} A^{[k,\ell]}_{k+1})
  \end{equation}
is of order $\mathcal O(\rho^{k+2})$. The last term on the right has $\Pi_\ell$-block-structure.

Finally we obtain a block-diagonalisation up to $\Pi_k$-structure with last remaining term 
$\tilde\Lambda_{k+1} = A^{[k,k]}_{k+1}$. If we assume that this matrix is diagonable and denote a corresponding diagonaliser as $\tilde M_k$, it follows that
\begin{equation}
   M_0 (I+\rho M^{(1)}) \tilde M_1 \cdots (I+\rho^{k+1} K^{[k]}_0) \cdots (I+\rho K^{[k]}_k) \tilde M_k
\end{equation}
diagonalises $A(\rho)$ for small $\rho$ modulo $\mathcal O(\rho^{k+2})$.

\begin{rem}
If we had followed Remark~\ref{rem2.3} we could simplify this step (but with the cost of already determining the relevant matrices in the previous steps). If we always perform a perfect block-diagonalisation, the remainder term $\tilde A^{[k]}_{k+1}$ has $\Pi_{k-1}$-block-structure and can be treated with one diagonalisation hierarchy instead of $k+1$ sub-steps. 
\end{rem}

\subsubsection{A hierarchy of conditions and results}
We recall that in each step of the iterative scheme there appeared one assumption, namely that $\tilde\Lambda_{k+1}= A^{[k,k]}_{k+1}$ is diagonable. It is difficult to express these assumptions by conditions on the coefficients in the original expansion \eqref{eq:1}. We will abuse notation a little more and denote $\tilde\Lambda_0=A_0$. Then the following definition makes sense:

\begin{df}
The matrix family $A(\rho)$ satisfies Assumption {\bf (A$_n$)} if the matrices $\tilde\Lambda_k$  for $k=0,1,\ldots n$ are all diagonable.
\end{df}

\begin{prop}
If $A(\rho)$ satisfies {\bf (A$_n$)}, then there exists a small interval $I_n$ and a uniformly bounded family of invertible matrices $M_n(\rho)$, $\rho\in I_n$, with uniformly bounded inverse, such that
\begin{equation}
   M_n^{-1}(\rho)A(\rho)M_n(\rho) 
\end{equation}  
is diagonal modulo $\mathcal O(\rho^{n+1})$.
\end{prop}

In Definition~\ref{df21} we introduced the notion of a non-degenerate family of matrices. This definition can be relaxed by the following weaker notion. 

\begin{df}
The matrix family $A(\rho)$ is called {\em non-degenerate of order $n$} in $\rho=0$ if it satisfies  {\bf (A$_{n}$)}, the matrix $\tilde\Lambda_n$ consists of  non-degenerate blocks and this number $n$ is minimal.
\end{df}

Obviously  this definition coincides with the old one for $n=0$. Furthermore,  non-degeneracy of order $n$ implies  {\bf (A$_{n'}$)} for all $n'\in\mathbb N$ and allows for perfect diagonalisation by our scheme. 

\begin{thm}
Assume that $A(\rho)$ is non-degenerate of order $n$ in $\rho=0$. Then the statement of Theorem~\ref{thm21} holds true. Furthermore, at least two eigenvalues of $A(\rho)$ coincide modulo $\mathcal O(\rho^n)$ in $\rho=0$.
\end{thm}

\begin{rem}
The statements of Propositions~\ref{prop:22} and~\ref{prop:23} transfer with the respective change in notation, so knowing an approximation of the diagonaliser $M(\rho)$ allows us to estimate eigenvalues and eigenprojections of the matrix-valued function $A(\rho)$.
\end{rem}

\subsubsection{An example}\label{sec2.2a} We want to provide at least one detailed example to make the algorithm more comprehensible. For this we consider the matrix-valued function
\begin{equation}
 A(\rho) = A_0 + \rho A_1 + \rho^2 A_2
\end{equation}
with coefficient matrices
\begin{align}
  A_0 = \frac12\begin{pmatrix} \alpha&\alpha&0\\\alpha&\alpha&0\\0&0&0 \end{pmatrix}, \qquad
  A_1 = \begin{pmatrix}\beta & 0 & \gamma \\0 & -\beta& \gamma\\
\delta/2 &\delta/2 & 0 \end{pmatrix}, \qquad
  A_2 = \begin{pmatrix} 0&0&0\\0&0&0\\0&0& \kappa\end{pmatrix}.
\end{align}
They will reappear in Section~\ref{sec:3.4} within the treatment of a model of thermo-elasticity. 

{\sl Step 0.} The matrix $A_0$ is degenerate with eigenvalues are $2\alpha$ and $0$ twice. A diagonaliser of $A_0$ is given by
\begin{equation}
   M_0 = \begin{pmatrix}1 & 1 & 0 \\ 1 & -1 & 0 \\ 0 & 0 &1  \end{pmatrix},
   \qquad M_0^{-1} =  \frac12  \begin{pmatrix}1 & 1 & 0 \\ 1 & -1 & 0 \\ 0 & 0 &2  \end{pmatrix}.
\end{equation}
Hence 
\begin{equation}
   \tilde A^{[0]}(\rho) = \begin{pmatrix} \alpha & 0 &0 \\0&0& 0 \\ 0& 0& 0\end{pmatrix}
    + \rho \begin{pmatrix} 0& \beta & \gamma \\ \beta & 0 & 0 \\ \delta & 0 & 0 \end{pmatrix} + \rho^2 \begin{pmatrix} 0&0&0\\0&0&0\\0&0& \kappa\end{pmatrix}. 
\end{equation}

{\sl Step 1.}  For $\alpha\ne0$ we have $\Pi_0=(1,2)$ and the matrices $\tilde\Lambda_1$ and $K_0^{[0]}$ are given by
\begin{equation}
  \tilde\Lambda_1 = 0,\qquad
   K_0^{[0]} =  \begin{pmatrix} 0& -\frac\beta{\alpha} & -\frac \gamma{\alpha} \\ \frac\beta{\alpha} & 0 & 0 \\ \frac\delta{\alpha} & 0 & 0 \end{pmatrix}.
\end{equation}
Transforming with $I+\rho K_0^{[0]}$ (and noting that $\tilde\Lambda_1$ is already of diagonal form) yields
\begin{align}
    \tilde A^{[1]} (\rho) &=  (I-\rho K_0^{[0]} + \rho^2 (K_0^{[0]})^2) 
    \tilde A^{[0]}(\rho) (I+\rho K_0^{[0]}) + \mathcal O(\rho^3) \notag\\
    &=\Lambda_0 %+ \rho([\Lambda_0,K_0^{[0]}]+\tilde A_1^{[0]}) 
    + \rho^2 (\tilde A^{[0]}_2 + \tilde A^{[0]}_1 K_0^{[0]}% - K_0^{[0]}\Lambda_0 K_0^{[0]} - K_0^{[0]} \tilde A^{[0]}_1 + (K_0^{[0]})^2\Lambda_0 
    ) + \mathcal O(\rho^3) \notag\\
    &= \begin{pmatrix} \alpha & 0 &0 \\0&0& 0 \\ 0& 0& 0\end{pmatrix}
    + \rho^2 \begin{pmatrix}  \frac{\beta^2+\gamma\delta}{\alpha} & 0 &0 \\
    0 & -\frac{\beta^2}{\alpha} & -\frac{\beta\gamma}{\alpha} \\ 0&-\frac{\beta\delta}{\alpha} & \kappa-\frac{\gamma\delta}{\alpha} \end{pmatrix} + \mathcal O(\rho^3).
\end{align}
 
{\sl Step 2.} Because $\Lambda_1=0$ we get $\Pi_1=\Pi_0$ and the partition can not be refined. So the best we can do is to $(1,2)$-block-diagonalise the matrix $\tilde A^{[1]} (\rho)$ modulo $\mathcal O(\rho^3)$ in the first two sub-steps. Since it is already (1,2)-block-diagonal, we can skip these and proceed directly to the final one. The matrix $\tilde \Lambda_2= A^{[1,1]}_2=\tilde A^{[1]}_2$ 
is diagonable as soon as its lower right block is. So we calculate its eigenvalues. They are given by
\begin{subequations}
\begin{align}
\lambda_1^{[2]}& =  \frac{\beta^2+\gamma\delta}{\alpha}, \\
 \lambda_{2/3}^{[2]}&=\frac12 \left( \kappa - \frac{\beta^2 +\gamma\delta}{\alpha}\right) \pm 
   \sqrt{\frac14\left( \kappa - \frac{\beta^2 +\gamma\delta}{\alpha}\right)^2 + \frac{\beta^2\kappa}{\alpha}},
\end{align}
\end{subequations}
and the latter two are distinct provided that $(\beta^2+(\gamma\delta+\alpha\kappa))^2-4\alpha\kappa\gamma\delta\ne0$.
We assume this; then assumption {\bf (A$_2$)} is satisfied and the matrix family $A(\rho)$ is non-degenerate of order $2$ in $\rho=0$. 

It follows that eigenvalues and eigenprojections have full asymptotic expansions as $\rho\to0$, and the main terms can be read off from the above matrices, i.e., 
\begin{align}
   \lambda_1(\rho) = \alpha +\rho^2 \lambda_1^{[2]}+ \mathcal O(\rho^3),\qquad
   \lambda_{2/3}(\rho) = \rho^2 \lambda_{2/3}^{[2]} + \mathcal O(\rho^3).
\end{align}

\subsection{On the optimality of the conditions}\label{sec2.3}
We call a matrix family $A(\rho)$ uniformly diagonable for all $\rho$ if there exists a family of invertible matrices $T(\rho)$, continuous in $\rho$, such that $T^{-1}(\rho)A(\rho)T(\rho)$ is diagonal and the matrices satisfy a uniform bound $\sup_\rho\|T(\rho)\|<\infty$. 

\begin{lem}
Assume that the matrix family $A(\rho)$ is uniformly diagonable for all $\rho<\epsilon$ up to and including $\rho=0$. 
Then the assumptions  {\bf (A$_{n}$)} are satisfied for all $n$.
\end{lem}
\begin{proof}
If we plug in $\rho=0$ we obtain that $A_0$ is diagonable and thus   {\bf (A$_{0}$)} follows. Assume now that for one particular $n\ge0$ the assumption  {\bf (A$_{n}$)} holds, but  {\bf (A$_{n+1}$)} fails. Then we can apply the first $n$ iterations of our diagonalisation scheme and obtain the existence of a polynomial matrix function
$\tilde M_n(\rho)$, uniformly invertible for small $\rho$, such that
\begin{equation}
   \tilde M_n^{-1}(\rho) A(\rho) \tilde M_n(\rho) = \sum_{j=0}^n \rho^j\Lambda_j + \rho^{n+1} A^{[n,n]}_{n+1} + \mathcal O(\rho^{n+2})
\end{equation}
is valid. Furthermore, $A^{[n,n]}_{n+1}$ is not diagonable, i.e.,  one of its blocks is not diagonable.
Note that on the subspace $W$ corresponding to this block all diagonal matrices $\Lambda_j$ are just multiples of the identity. 

By assumption the left hand side is uniformly diagonable and thus there exists a transformation $\tilde T(\rho) = \tilde M_n^{-1}(\rho) T(\rho)$ diagonalising the left hand side. Restricting consideration to the above mentioned invariant subspace $W$ gives the diagonal matrix
\begin{equation} 
 (\sum_j \rho^j \lambda_i^{[j]}) I + \rho^{n+1} \tilde T^{-1}(\rho)  A^{[n,n]}_{n+1}  \tilde T(\rho)\bigg|_W +  \mathcal O(\rho^{n+2}).
\end{equation}
Since the first addend is diagonal it remains to consider the last two. Dividing by $\rho^{n+1}$ and taking the limit for $\rho\to0$ (which exists due to continuity of $T(\rho)$) gives that $\tilde T(0)\big|_W$ diagonalises the non-diagonable block and thus the desired contradiction.
\end{proof}

\begin{rem} 
From classical perturbation theory, see e.g. the book of Kato, \cite{Kato:1980}, or Knopp, \cite{Knopp:1952}, it is clear that analyticity of $A(\rho)$ implies that the eigenvalues are branches of algebraic functions which have Puiseux series containing fractional exponents (with powers $\rho^{k/p}$, $k=0,1,\ldots$, where $p$ corresponds to the size of irreducible groups of eigenvalues --so called $\lambda$-groups-- that are permuted if the degenerate point $0$ is encircled in the complex plane). If we assume that fractional powers do appear and no eigenvalues coincide identically (i.e.,  if we assume that the matrices are not permanently degenerate), then diagonalisers exist in a neighbourhood of $\rho=0$, but  due to the above lemma we know that they cannot be uniformly bounded / continuous in $\rho$.
\end{rem}

\begin{rem}
In Kato~\cite[Chapter II]{Kato:1980} asymptotic expansions of eigenvalues and eigenprojections corresponding to $\lambda$-groups of eigenvalues are discussed and first terms are given. The approach used there differs from our treatment and is based on Dunford integrals for resolvents
and uses analytic dependence of the matrix family on the involved parameters.
\end{rem}

\section{Applications}\label{sec3}
Power stems from flexibility. We will show by a selection of applications to what extent the schemes introduced in Section~\ref{sec2} can be adapted to deal with real problems in the analysis of (partial) differential equations without loosing their constructiveness. The selection is not complete, but intended to give an impression of the variety of possible uses. 

\subsection{Hyperbolic polynomials} As a first application of the standard scheme of Section~\ref{sec2.1} we want to discuss the behaviour of the roots of hyperbolic polynomials for large spatial frequencies
$\xi$.

\begin{df}
A polynomial $L(\tau,\xi)$ in $1+n$ variables and of degree $m$ is called {\em strictly hyperbolic} if its $m$-homogeneous part $L_m(\tau,\xi)$ seen as polynomial in $\tau$ parameterised by $\xi$ has $m$ distinct real roots $\phi_1(\xi),\dots,\phi_m(\xi)$. 
\end{df}

We are interested in the behaviour of the roots of $L(\tau,\xi)$, the so-called characteristic roots, as $|\xi|\to\infty$. In particular, we want to prove the following asymptotic expansion:

\begin{thm}\label{thm3.1}
The characteristic roots of a strictly hyperbolic polynomial have full asymptotic expansions as
$|\xi|\to\infty$, i.e., 
\begin{equation}\label{eq:3.1}
  \tau_j(\xi)\sim|\xi| \phi_j(\eta) + \tau_j^{(0)}(\eta)+|\xi|^{-1}\tau_j^{(1)}(\eta)+\cdots,\qquad |\xi|\to\infty
\end{equation}
uniformly in $\eta=\xi/|\xi|\in\mathbb S^{n-1}$ with smooth (algebraic) functions $\tau_j^{(k)} : \mathbb S^{n-1}\to\C$.
\end{thm}
\begin{proof}[Sketch of proof]
The statement follows from the argument used to prove Theorem~\ref{thm21} after rewriting it in terms of matrices and recognising that everything works uniformly in parameters. For this we rewrite the polynomial as 
\begin{equation}
  L(\tau,\xi)=c \sum_{k=0}^m \tau^{m-k} |\xi|^k p_k(\xi),
\end{equation}
with $|\xi|^k p_{k}(\xi)$ polynomial in $\xi$ of degree (at the most) $k$ and $p_0(\xi)=1$. 
Without loss of generalisation we set $c=1$ and form a companion matrix with $L(\tau,\xi)$ as characteristic polynomial,
\begin{equation}\mathcal L(\xi) = |\xi|
  \begin{pmatrix} 
  & 1 & & \\
  & & 1 & \\
  & & & \ddots & \\
  & & & & 1 \\
  -p_m(\xi) & -p_{m-1}(\xi) & \cdots & \cdots & -p_1(\xi)     
  \end{pmatrix}\in \C^{m\times m}.
\end{equation}
This matrix can be written as sum of homogeneous components 
\begin{equation}
\mathcal L(\xi) = |\xi| \mathcal L_{0}(\eta) + \mathcal L_1(\eta) +\cdots+|\xi|^{1-m} \mathcal L_{m}(\eta),
\qquad\eta=\xi/|\xi|\in\mathbb S^{n-1},
\end{equation}
corresponding to the homogeneous parts of the polynomials $|\xi|^kp_k(\xi)$. The assumption of strict hyperbolicity is equivalent to the fact that the matrix $\mathcal L_0(\eta)$ has $m$ distinct real eigenvalues $\phi_1(\eta), \dots,\phi_m(\eta)$ (which are uniformly separated by the compactness of $\mathbb S^{n-1}$) and the standard diagonalisation scheme applied for $\rho=|\xi|^{-1}$ gives representations for all functions involved in \eqref{eq:3.1}.
\end{proof}

We see that the terms $\tau_j^{(0)}(\eta)$ up to $\tau_j^{(k-1)}(\eta)$ depend {\em only} on the
homogeneous components $\mathcal L_0$ up to $\mathcal L_k$. Therefore, the following conclusion
is apparent. 
\begin{cor}\label{cor3.2}
If the polynomial $L(\tau,\xi)-L_m(\tau,\xi)$ is of degree $L<m-1$, then the coefficient functions
$\tau_j^{(0)}(\eta)$ up to $\tau_j^{(m-L-2)}(\eta)$ in the above given asymptotic expansion vanish identically.
\end{cor}
The statements of Theorem~\ref{thm3.1} and Corollary~\ref{cor3.2} are of particular interest in the situation of \cite{Ruzhansky:2007}, where dispersive estimates for higher-order hyperbolic equations with constant coefficients are discussed. 

\subsection{Asymptotic integration of systems of differential equations} Assume we are given a linear system of differential equations
\begin{equation}\label{eq:3.5}
  \dot v(t) = \frac{\mathrm  dv}{\mathrm d t} = A(t) v,\qquad v(0)=v_0\in \C^m,
\end{equation}
with a time-dependent coefficient matrix $A(t)\in C^\infty(\R;\C^{m\times m})$ having an asymptotic expansion 
\begin{equation}
   A(t) = A_0 + t^{-1} A_1 + t^{-2} A_2 + \cdots ,\qquad t\to\infty,
\end{equation} 
with non-degenerate $A_0$. We assume further that derivatives of $A(t)$ also have asymptotic expansions (which in consequence implies that we can differentiate the above expansion term by term).

\begin{thm}\label{thm:3.2} Assume $A_0$ is non-degenerate.
There exists a uniformly bounded and invertible matrix function $M(t)$ and a diagonal matrix function $\Lambda(t)$, both having full asymptotic expansions as $t\to\infty$,  such  that the operator identity 
\begin{equation}\label{eq:3.7}
    \big( \frac{\mathrm  d}{\mathrm d t} - A(t)\big) M(t) = M(t) \big(\frac{\mathrm  d}{\mathrm d t} -\Lambda(t)\big)\quad \mod\bigcap_N\mathcal O(t^{-N})
\end{equation}
holds modulo matrices decaying faster than all polynomials. 
\end{thm}
\begin{proof}[Sketch of proof]
We apply a variant of the standard diagonalisation scheme for diagonalising this system. Thus, we  construct recursively matrices 
\begin{subequations}
\begin{align}
M_k(t)&= M_0(I+t^{-1}M^{(1)}+t^{-2}M^{(2)}+\cdots+t^{-k}M^{(k)}),\\
\Lambda_k(t)&= \Lambda_0+t^{-1} \Lambda^{(1)} +t^{-2}\Lambda^{(2)} + \cdots+t^{-k}\Lambda^{(k)},
\end{align}
\end{subequations}
and follow the two steps from the proof of Theorem~\ref{thm21}. Again $M_0$ is the diagonaliser of $A_0$ and {\sl Step 1} transfers directly. In {\sl Step 2} we include the differential operator in 
\eqref{eq:2.5} 
\begin{equation}%\label{eq:2.5}
   B^{(k)}(t) =  \big(\frac{\mathrm  d}{\mathrm d t}-A(t)\big) M_{k-1}(t)-M_{k-1}(t)\big(\frac{\mathrm  d}{\mathrm d t}- \Lambda_{k-1}(t)\big)  = \mathcal O(t^{-k})
\end{equation}
and, thus, view it also as an operator identity. Analogously\footnote{Since we are concerned with $-A(t)$ in formula \eqref{eq:3.7} we include a minus sign in the definition of $\Lambda^{(k)}$.}, we define the matrices $\tilde B^{(k)}=\lim_{t\to\infty} t^k M_0^{-1} B^{(k)}(t)$, $\Lambda^{(k)}=-\diag \tilde B^{(k)}$ and $M^{(k)}$ as solution to the commutator equation $[\Lambda_0,M^{(k)}]=\tilde B^{(k)}+\Lambda^{(k)}$.  Since time-derivatives are one order better, the scheme works through,
\begin{align}
  B^{(k+1)}(t) =& B^{(k)}(t) \notag\\ 
  &- t^{-k} \big(A(t) M_0 M^{(k)} -M_0 M^{(k)}\Lambda_k(t) -M_{k-1}(t) \Lambda^{(k)}\big) \\
  &- kt^{-k-1} M_0 M^{(k)} \notag
\end{align}
is of order $\mathcal O(t^{-k-1})$, and the recursion provides all matrices involved in the statement together with the necessary bounds.
\end{proof}

\begin{rem}
If the matrix $A_0$ is degenerate, multi-step schemes apply in a similar way. 
\end{rem}

A particular application of Theorem~\ref{thm:3.2} is the derivation of WKB approximations of solutions to hyperbolic systems.

\begin{cor}
Assume in addition that $A_0$ is skew. Then the solutions to \eqref{eq:3.5} satisfy
\begin{equation}
  v(t)  =  M(t) \exp\left(\int_0^t \Lambda(s)\mathrm ds\right) Q(t) M^{-1}(t) v_0,
\end{equation}
with a uniformly bounded matrix $Q(t)$ converging faster than polynomially to an invertible limit $Q_0=Q(\infty)$ as $t\to\infty$.
\end{cor}
\begin{proof}
We denote the remainder in \eqref{eq:3.7} as
\begin{equation}
   R(t) = M^{-1}(t) \dot M(t) - M^{-1}(t)A(t)M(t)+\Lambda(t) \in\bigcap_N \mathcal O(t^{-N}),
\end{equation}
such that $v^{(1)}(t) = M^{-1}(t) v(t)$ satisfies $\dot  v^{(1)}(t) = \big(\Lambda(t) + R(t) \big) v^{(1)}(t)$. The diagonal part of this remaining system can be solved directly by means of the fundamental solution
\begin{equation}
\mathcal E(t) = \exp\left(\int_0^t \Lambda(s)\mathrm ds\right).
\end{equation}
To treat $R(t)$ as perturbation we make an Ansatz for the fundamental solution of the diagonalised system of the form  $\mathcal E(t) Q(t)$. This yields for $Q(t)$ 
\begin{equation}
 \frac{\mathrm  d}{\mathrm d t} Q(t) = \mathcal R(t)Q(t) = \big( \mathcal E^{-1}(t)R(t)\mathcal E(t)\big)Q(t),\qquad Q(0)=I,
\end{equation}
which can be solved in terms of the Peano-Baker series
\begin{equation}
   Q(t)=I+\sum_{k=1}^\infty \int_0^t \mathcal R(t_1)\int_0^{t_1} \mathcal R(t_2) \cdots\int_0^{t_{k-1}}
   \mathcal R(t_k)\mathrm d t_k\cdots \mathrm d t_2\mathrm d t_1.
\end{equation}
If $A_0$ is skew it follows that $\mathcal E(t)$ and $\mathcal E^{-1}(t)$ both satisfy polynomial bounds. Hence, $\mathcal R(t)$ decays fast in $t$, and the statement follows from the estimates
\begin{align}
   \|Q(t)\| &\le \exp\left(\int_0^t \mathcal R(s)\mathrm ds\right),\\
   \|Q(t)-Q(\infty)\| &\le \int_t^\infty \mathcal R(s)\mathrm ds \;\exp\left(\int_0^\infty \mathcal R(s)\mathrm ds\right) \in\bigcap_N \mathcal O(t^{-N}),
\end{align}
in combination with Liouville's theorem, $\det Q(\infty) = \exp( \int_0^\infty \mathrm{trace}\, R(s)\mathrm ds)\ne0$.
\end{proof}

A variant of this approach was used in \cite[Chapter~2]{Yagdjian:1997}.
For a slightly different diagonalisation based method using less regularity of the coefficient see the treatise of Eastham, \cite[Chapter~1]{Eastham:1989}, or the second author's utilisation in \cite{Wirth:2007d}. 

\subsection{Diagonalisation within symbolic hierarchies}\label{sec3.2}
The method of the previous section can be extended to systems involving parameters. Typical applications are hyperbolic partial differential equations with $t$-dependent coefficients, which are treated by means of the partial Fourier transform. In this case one has to be careful by choosing $t$ large in dependence of the parameter and this leads to the introduction of so-called zones. We will not go into details here, but refer the reader to the fundamental treatise of Yagdjian on weakly hyperbolic problems, \cite[Chapter~3]{Yagdjian:1997}, and applications deducing dispersive estimates for wave equations with variable propagation speed of Reissig and co-authors, \cite{Reissig:2000}, \cite{Reissig:2005}. More involved considerations including several zones and different diagonalisation hierarchies turn up for the treatment of lower-order terms, e.g. in \cite{Hirosawa:2003}, \cite{Reissig:2006}, \cite{Wirth:2006}
 and \cite{Wirth:2007} .

A second modification of the standard scheme is to replace the matrices with complex entries by matrix-valued pseudo-differential operators and to apply it within certain symbolic hierarchies. This was developed in \cite{Yagdjian:1997} in order to prove well-posedness of  hyperbolic systems with multiple characteristics, and similarly by Kubo-Reissig \cite{Kubo:2003}, \cite{Kubo:2003a}, and Hirosawa-Reissig \cite{Hirosawa:2004a}, to obtain corresponding results for strictly hyperbolic problems with certain non-Lipschitz coefficients.

We will sketch the approach in a somewhat simplified case. We denote by $\mathcal Z_{hyp}(N)$ and $\mathcal Z_{pd}(N)$  the subsets
\begin{subequations}
\begin{align}
\mathcal Z_{hyp}(N) &= \{ (t,x,\xi)\; :\; t \langle\xi\rangle \ge N \},\\
\mathcal Z_{pd}(N) &= \{ (t,x,\xi)\; :\; t \langle\xi\rangle \le N \}
\end{align}
\end{subequations}
of the extended phase space $(0,T]\times\R^n_x\times\R^n_\xi$, where as usual $\langle\xi\rangle=\sqrt{1+|\xi|^2}$, and consider pseudo-differential operators corresponding to the following symbol classes:
\begin{df}
A symbol $a(t,x,\xi)\in C^\infty((0,T]\times\R^n_x\times\R^n_\xi)$ belongs to $\mathcal S_N\{m_1,m_2\}$ if the symbolic estimates
\begin{subequations}
\begin{align}
   \sup_{(t,x,\xi)\in \mathcal Z_{hyp}(N)} \left| \partial_t^k\partial_\xi^\alpha\partial_x^\beta a(t,x,\xi)\right|
   &\le C_{k,\alpha,\beta} \langle\xi\rangle^{m_1-|\alpha|} \left(\frac1t\right)^{m_2+k}, \\
    \sup_{(t,x,\xi)\in \mathcal Z_{pd}(N)} \left|\partial_\xi^\alpha\partial_x^\beta a(t,x,\xi)\right|
   &\le C'_{\alpha,\beta} \langle\xi\rangle^{m_1+m_2-|\alpha|}
\end{align}
\end{subequations}
hold true for all multi-indices $\alpha,\beta \in\mathbb N_0^n$ and all $k\in\mathbb N_0$.
\end{df}

For a full account of calculus properties of such kinds of symbol classes we refer to \cite{Yagdjian:1997},
we only mention embedding properties and relations to classical symbols here. Note that $\langle\xi\rangle\ge 1$ and $t^{-1}\ge T^{-1}$, and so
\begin{subequations}
\begin{align}\label{eq:3.19a}
\mathcal S_N\{m_1-k,m_2-\ell\}\hookrightarrow \mathcal S_N\{m_1,m_2\},\qquad k,\ell\ge0.
\intertext{Furthermore, the definition of the hyperbolic zone implies}
\mathcal S_N\{m_1-k,m_2+k\}\hookrightarrow \mathcal S_N\{m_1,m_2\},\qquad k \ge 0
\end{align}
\end{subequations} 
The embedding hierarchy \eqref{eq:3.19a} with $\ell=0$ will be denoted as calculus hierarchy. It allows us to transfer the usual symbolic calculus and corresponding asymptotic expansions from the H\"ormander classes 
$S^{m_1}_{1,0} (\R^n_x\times\R^n_\xi)$ via the embedding $\mathcal S_N\{m_1,m_2\}\hookrightarrow C^\infty((0,T],  S^{m_1}_{1,0} (\R^n_x\times\R^n_\xi))$. Thus, the composition of operators corresponds to the usual Leibniz product $\sharp$ of symbols. Symbols from $\mathcal S_N\{m_1,m_2\}$ which are invertible modulo smoothing operators from $\mathcal S_N\{-\infty,m_2\}$ with a parametrix in $\mathcal S_N\{-m_1,-m_2\}$ will be called elliptic, and the parametrix is denoted with an upper $\sharp$. 
The second hierarchy will be used within the diagonalisation scheme, and we denote the corresponding residual class of this hierarchy by $\mathcal H_N\{m\}=\bigcap_{k} \mathcal S_N\{m-k,k\}$. 

Assume for the following that we have a matrix-valued symbol $A(t,x,\xi)\in\mathcal S_N\{1,0\}$, whose eigenvalues satisfy the non-degeneracy assumption
\begin{equation}\label{eq:3.19}
  | \lambda_i(t,x,\xi) -\lambda_j(t,x,\xi) | \ge C \langle\xi\rangle 
\end{equation}
uniformly in $i\ne j$ and $(t,x,\xi)\in \mathcal Z_{hyp}(N)$ for a certain large $N$.

\begin{thm}\label{thm:3.5}
There exists an elliptic matrix valued symbol $M(t,x,\xi)\in\mathcal S_N\{0,0\}$ and a diagonal symbol $\Lambda(t,x,\xi)\in\mathcal S_N\{1,0\}$ such that the operator identity
\begin{equation}
    \big(\frac{\mathrm d}{\mathrm dt} - A(t,x,\xi) \big) \sharp M(t,x,\xi) = M(t,x,\xi) \sharp   \big(\frac{\mathrm d}{\mathrm dt} - \Lambda(t,x,\xi) \big) \quad\mod\mathcal H_N\{1\}
\end{equation}
holds modulo symbols from the residual class $\mathcal H_N\{1\}$. 
\end{thm}

\begin{proof}[Sketch of proof]
The proof transfers almost word by word from Theorems~\ref{thm21} and \ref{thm:3.2}, except that we now replace multiplications by the Leibniz product~$\sharp$ and solve the Sylvester equation \eqref{eq:2.6} only modulo symbols of lower order (such that \eqref{eq:2.7} still remains true).

For completeness we give the main steps and the corresponding symbol estimates. We construct
\begin{subequations}\label{eq:3.22}
\begin{align}
M_k(t,x,\xi)&= M_0(t,x,\xi)\sharp\left(I+\sum_{j=1}^k M^{(j)}(t,x,\xi)\right)\\
\intertext{and}
\Lambda_k(t,x,\xi)&= \Lambda_{0}(t,x,\xi) +\sum_{j=1}^k \Lambda^{(j)}(t,x,\xi)
\end{align}
\end{subequations}
with matrix-valued symbols satisfying $M_0\in \mathcal S_N\{0,0\}$ elliptic, $\Lambda_0\in\mathcal S_N\{1,0\}$, $M^{(j)}\in\mathcal S_N\{-j,j\}$ and $\Lambda^{(j)}\in\mathcal S_N\{1-j,j\}$. 

In {\sl Step 1} we choose $M_0(t,x,\xi)$ to be a diagonaliser of the full symbol $A(t,x,\xi)$ within 
$\mathcal Z_{hyp}(N)$, which is uniformly bounded and invertible. This can be done by \eqref{eq:3.19}. Furthermore, we define $\Lambda_0=\diag M_0^\sharp\sharp A\sharp M_0$.   In {\sl Step 2} we proceed recursively and, in analogy to \eqref{eq:2.5}, define for $k=1,2,\ldots$
\begin{equation}
   B^{(k)}(t,x,\xi) =  \big(\frac{\mathrm  d}{\mathrm d t}-A(t,x,\xi)\big)\sharp M_{k-1}(t,x,\xi)-M_{k-1}(t,x,\xi)\sharp\big(\frac{\mathrm  d}{\mathrm d t}- \Lambda_{k-1}(t,x,\xi)\big) 
\end{equation}
For $k=1$ we know from Step 1 that $B^{(1)}\in\mathcal S\{0,1\}$. For our recursive argument we assume $\tilde B^{(k)}=M_0^{\sharp}\sharp B^{(k)}\in \mathcal S_N\{1-k,k\}$ and set $\Lambda^{(k)}=-\diag  \tilde B^{(k)}$. Defining $M^{(k)}$ by \eqref{eq:2.7}
\begin{equation}
   \big(M^{(k)}(t,x,\xi)\big)_{ij} = \begin{cases} \frac{\big(\tilde B^{(k)}(t,x,\xi)\big)_{ij}}{\lambda_{i}(t,x,\xi)-\lambda_{j}(t,x,\xi)}\chi_N(t,x,\xi), \qquad & i\ne j,\\0, \qquad & i=j, \end{cases}
\end{equation}
with $\chi_N(t,x,\xi)$ a smooth cut-off function localising to $\mathcal Z_{hyp}(N)$, implies
\begin{align}
   [\Lambda_0,M^{(k)}] &= \Lambda_0\sharp M^{(k)}-M^{(k)}\sharp \Lambda_0\notag\\& = \tilde B^{(k)}+\Lambda^{(k)} \mod \mathcal S_N\{-k,k\},
\end{align}
so
\begin{align}
   B^{(k+1)} &= B^{(k)}- A\sharp M_0\sharp M^{(k)}+M_0\sharp M^{(k)}\sharp\Lambda_0+M_0\sharp\Lambda^{(k)}\notag\\
   &\in M_0\sharp \left( \tilde B^{(k)} -[\Lambda_0,M^{(k)}] + \Lambda^{(k)} \right) + \mathcal S_N\{-k,k+1\}\\
   &\in \mathcal S_N\{-k,k+1\}\notag.
\end{align}
We can proceed inductively to construct all $M_k$, $\Lambda_k$ and the statement itself follows by forming asymptotic sums instead of the finite ones in \eqref{eq:3.22}.
\end{proof}

This statement may be used to deduce well-posedness of degenerate Cauchy problems. Assume that the matrix $A(t,x,\xi)\in\mathcal S_N\{1,0\}$ is hyperbolic in the sense that $|\mathrm{Re}\,\lambda_j(t,x,\xi)|\le C$ uniform in $(t,x,\xi)\in(0,T]\times\R_x^n\times\R_\xi^n$ in combination with \eqref{eq:3.19}, but degenerates as $t$ approaches $0$,
\begin{equation}
 \|A(t,x,\xi)\|\lesssim 1,\qquad  \| \partial_t A(t,x,\xi) \| \approx \frac1t.
\end{equation} 
We may ask whether we can pose the Cauchy problem 
\begin{equation}
  \frac{\mathrm d}{\mathrm dt} v(t,x) = A(t,x,\mathrm D) v(t,x),\qquad v(0,\cdot)\in L^2(\R^n;\C^m).
\end{equation}
The standard energy argument does not work, because this would mean that we have to differentiate the full symbol $A(t,x,\xi)$. But the above diagonalisation argument simplifies the problems substantially. 
If we consider $v^{(1)}(t,x) = M^\sharp(t,x,\mathrm D) v(t,x)$ and apply Theorem~\ref{thm:3.5}, we obtain\begin{equation}
  \frac{\mathrm d}{\mathrm dt} v^{(1)}(t,x) = \big(\Lambda(t,x,\mathrm D) + R(t,x,\mathrm D) \big)v^{(1)}(t,x),
\end{equation}
with a remainder $R(t,x,\xi)\in\mathcal H_N\{1\}$. The diagonal part of this system can be solved in terms of a diagonal matrix of (elliptic) Fourier integral operators of order zero. This follows from the hyperbolicity assumption. Furthermore, it can be shown that $\mathcal H_N\{m\}$ is invariant under conjugations with such Fourier integral operators (see \cite[Prop.~3.8.12]{Yagdjian:1997}), and for an arbitrary $P\in \mathcal H_N\{1\}$ the operator $\frac{\mathrm d}{\mathrm dt} -P(t,x,\mathrm D)$ has a pseudo-differential fundamental solution in $C([0,T],\Psi^0(\R^n))$, $0$ included (see \cite[Prop.~3.9.1]{Yagdjian:1997}). Well-posedness follows.

The argument can be found, in full detail and including fast oscillations within the coefficient (i.e.,  a further $\log$-term in the estimates), in \cite{Kubo:2003}.

\subsection{Pseudo-differential decoupling of hyperbolic-parabolic coupled systems}  The two-step diagonalisation scheme appeared first in Wang \cite{Wang:2003}, \cite{Wang:2003a} in applications decoupling hyperbolic-parabolic coupled systems, and was used to study the propagation of singularities, \cite{Reissig:2005a}. In  \cite{Wirth:2007b},  \cite{Wirth:2007c} it was an essential tool for understanding dispersive estimates for anistropic thermo-elasticity in two space dimensions.

Again we will give a somewhat simplified example, which shows the main arguments involved and the importance of the two-step procedure. For this we choose a thermo-elastic system,
\begin{subequations}
\begin{align}
   &U_{tt} - (2\mu+\lambda) \nabla\nabla\cdot U + \mu \nabla\times(\nabla\times U) + \gamma_1 \nabla\theta = 0,\\
   &\theta_t - \beta^2\Delta\theta +\gamma_2 \nabla\cdot U_t = 0,
\end{align}
\end{subequations}
with smooth scalar functions $\lambda,\mu,\beta,\gamma_1,\gamma_2\in C^\infty(\R_t\times\Omega)$,
$\Omega\subset\R^3$, subject to the restrictions $\mu>0$, $2\mu+\lambda>0$, $\beta>0$ and $\gamma_1\gamma_2>0$. If one employs a Helmholtz decomposition of the displacement vector $U$,
\begin{equation}
   U=U^{p}+U^{s},\qquad \nabla\times U^{p}=0,\quad \nabla\cdot U^{s}=0,
\end{equation}
either globally on $\R^3$ or locally on any small open subset $\Omega'$ of the underlying domain $\Omega$ with trivial second cohomology class, the system simplifies. The solenoidal part $U^{s}$ satisfies the (well-studied) wave equation $U^s_{tt}-\mu\Delta U^s=0$, while the potential part $U^p=\nabla V$ satisfies 
\begin{subequations}
\begin{align}
& \nabla  V_{tt} -(2\mu+\lambda)\nabla \Delta V + \gamma_1\nabla\theta=0,\\
& \theta_t - \beta^2 \Delta\theta + \gamma_2\Delta V_t = 0.
\end{align}
\end{subequations}
We can write this as a first order $3\times3$ pseudo-differential system of a particular structure. We choose for this as new unknown $W=(\sqrt{-\Delta}V_t, \Delta V, \theta)^T$ such that
\begin{equation}
\frac{\mathrm d}{\mathrm d t} W = \begin{pmatrix}
    	0&\mathcal R^{-1} (2\mu+\lambda) \nabla & \mathcal R^{-1}\gamma_1 \nabla  \\
	-\sqrt{-\Delta} & 0 & 0 \\
	 \gamma_2 \sqrt{-\Delta} & 0 & \beta^2\Delta 
   \end{pmatrix} W =  A(t,x,\mathrm D)W,
\end{equation}
with $\mathcal R: \sqrt{-\Delta} V \mapsto \nabla V$ the/a Riesz transform (cf. \cite[Chapter III]{Stein:1970} for $\Omega=\R^n$ or use any invertible pseudo-differential operator with principal symbol $\xi/|\xi|$ if $\Omega\ne\R^n$). The symbol $A(t,x,\xi)$ belongs to the class $C^\infty(\R, S^{2}_{1,0}(\mathrm T^*\Omega'))$, but the only second order entry is at the lower right corner. 

An adaptation of the two-step scheme of \cite{Wang:2003} allows one to construct an elliptic matrix-valued pseudo-differential operator $M\in C^\infty(\R,S^{0}_{1,0}(\mathrm T^*\Omega'))$, such that
\begin{equation}
   \left(\frac{\mathrm d}{\mathrm dt} - A(t,x,\xi)\right)\sharp M(t,x,\xi) = M(t,x,\xi) \sharp \left(\frac{\mathrm d}{\mathrm dt} - A_1(t,x,\xi)\right)
\end{equation}
holds modulo $C^\infty(\R,S^{-\infty}_{1,0}(\mathrm T^*\Omega'))$, where $A_1=\bdiag(A_{11},a_{12})$ consists of a $2\times 2$-block $A_{11}\in C^\infty(\R,S^{1}_{1,0}(\mathrm T^*\Omega'))$ and a scalar entry $a_{12}\in C^\infty(\R,S^2_{1,0}(\mathrm T^*\Omega'))$. The $2\times 2$ block is strictly hyperbolic, and can be studied further by diagonalisation, i.e.,  there exists a second diagonaliser $N(t,x,\xi)\in C^\infty(\R,S^{0}_{1,0}(\mathrm T^*\Omega'))$ such that
\begin{equation}
   \left(\frac{\mathrm d}{\mathrm dt} - A_1(t,x,\xi)\right)\sharp N(t,x,\xi) = N(t,x,\xi) \sharp \left(\frac{\mathrm d}{\mathrm dt} - A_2(t,x,\xi)\right)
\end{equation} 
holds modulo $C^\infty(\R,S^{-\infty}_{1,0}(\mathrm T^*\Omega'))$, with a diagonal matrix $A_2=\diag(\tilde a_{21},\tilde a_{22}, a_{12})$  having entries $\tilde a_{2i}\in C^\infty(\R,S^{1}_{1,0}(\mathrm T^*\Omega'))$, $i=1,2$, and $a_{12}$ as above. Knowing the principle part of these scalar symbols allows to describe the propagation of singularities.

%For further details, see \cite{Wang:2003}, \cite{Wang:2003a}.

\subsection{Dispersive estimates and diffusive structure for thermo-elastic models}\label{sec:3.4}
Multi-step schemes have been developed in \cite{Jachmann:2008} in order to study models of thermo-elasticity with additional terms of lower order. They have been used in \cite{Wirth:2007b}, \cite{Wirth:2007c} and will also be central to the understanding of the generalisation of these results to higher dimensions. 

Again we will provide an application. It is taken from \cite{Jachmann:2008} and concerned with the derivation of dispersive estimates for one-dimensional thermo-elastic systems with terms of lower order. 
We restrict ourselves to one model, classical one-dimensional thermo-elasticity with an additional damping term:
\begin{subequations}\label{eq:3.36}
\begin{align}
   &u_{tt} - \tau^2 u_{xx} + \gamma_1\theta_x +m u_t =0, \\
   &\theta_t-\kappa\theta_{xx} +\gamma_2 u_{tx}=0, 
\end{align}
\end{subequations}
together with initial conditions $u(0,\cdot)=u_1$, $u_t(0,\cdot)=u_2$ and $\theta(0,\cdot)=\theta_0$. 
Parameters $\tau$, $\kappa$, $\gamma_1$, $\gamma_2$ and $m$ are assumed to be positive constants. Applying a partial Fourier transform with respect to the $x$-variable reduces the system to a system of ordinary differential equations, which can be written as system of first order in $V=(\hat u_+,\hat u_-,\hat\theta)^T$, $\hat u_\pm = \hat u_t \pm\mathrm i\tau\xi\hat u$. A short calculation gives
\begin{equation}
  \frac{\mathrm d}{\mathrm dt} V = A(\xi)V = \big( A_0 + \xi A_1 + \xi^2 A_2\big) V,
\end{equation}
with matrices
\begin{align}
  A_0 = \frac12 \begin{pmatrix} -m&-m&0\\-m&-m&0\\0&0&0 \end{pmatrix}, \qquad
  A_1 =\mathrm i \begin{pmatrix}\tau & 0 &\gamma_1 \\0 & -\tau& \gamma_1\\
 \frac{\gamma_2}2 & \frac{\gamma_2}2 & 0 \end{pmatrix}, \qquad
  A_2 = \begin{pmatrix} 0&0&0\\0&0&0\\0&0& -\kappa\end{pmatrix}.
\end{align}
If we assume for simplicity that $A(\xi)$ has no multiple eigenvalues, we can represent solutions to this system as a sum
\begin{equation}
   V (t,\xi) = \sum_{\nu(\xi)\in\spec A(\xi)} \mathrm e^{t \nu(\xi)} P_{\nu(\xi)} V_0
\end{equation}
over the spectrum of the matrix $A(\xi)$. In order to understand properties of solutions such as dispersive estimates or descriptions of asymptotic profiles it is enough to calculate the eigenvalues $\nu(\xi)$,
or at least to describe their main properties. 

This can be done in three steps. First, a consideration of the characteristic polynomial of $A(\xi)$ implies that no purely imaginary eigenvalues occur for $\xi\ne0$. Thus solutions with bounded frequencies away from $0$ either decay exponentially or increase exponentially. In a second step one can diagonalise and determine asymptotic expansions of eigenvalues (and eigenprojections) as $\xi\to0$ and $\xi\to\infty$. This proves that $\mathrm{Re}\,\nu(\xi)\le 0$ everywhere and gives enough information about the behaviour of the eigenvalues as $\xi\to0$ to determine decay estimates in a third step. 

We will not give the precise calculations, for which see \cite{Jachmann:2008}. It is easily seen that $A(\xi)$ is non-degenerate of order $1$ as $\xi\to\infty$, and, following the first two steps of the procedure from Section~\ref{sec2.2}, the eigenvalues have asymptotic expansions of the form
\begin{align}
   \nu_{par}(\xi) &= -\kappa \xi^2 + \frac{\gamma_1\gamma_2}2 + \mathcal O(\xi^{-1}),\\
   \nu_{hyp,\pm}(\xi) & = \pm \mathrm i \tau\xi - \frac{\gamma_1\gamma_2}{2\kappa} + \mathcal O(\xi^{-1}).
\end{align}
Thus, their real part is uniformly negative for large $\xi$, which implies exponential decay for large frequencies. It remains to consider $\xi\to0$. This calculation has already been done in Section~\ref{sec2.2a}, where it was also seen that $A(\xi)$ is non-degenerate of order 2 (under the above restrictions on the parameters). It follows that the three eigenvalues satisfy
\begin{equation}
   \nu_{0}(\xi) = -m  + \lambda_0 \xi^2 + \mathcal O(\xi^{3}),\qquad
   \nu_{\pm}(\xi)  =  - \lambda_\pm \xi^2 + \mathcal O(\xi^{3}),
\end{equation}
with positive constants $\lambda_0$, $\lambda_\pm$ depending on the given parameters. Thus,
one mode also leads to exponential decay, while the other two resemble a parabolic-type behaviour close to the corresponding heat equations $v_t = \lambda_\pm v_{xx}$. 

The above mentioned properties of the eigenvalues of $A(\rho)$ in combination with boundedness properties of the Fourier transform and H\"older inequality imply directly the following a-priori estimate. Again the detailed proof can be found in \cite{Jachmann:2008}.

\begin{thm} 
The solutions to the the system \eqref{eq:3.36} satisfy the a-priori estimate
\begin{equation}
   \| (u_t, u_x, \theta)(t,\cdot) \|_q \le C (1+t)^{-\frac12(\frac1p-\frac1q)} \big( \|u_1\|_{H^{p,r+1}} + 
   \| u_2,\theta_0\|_{H^{p,r}}\big)  
\end{equation}
for all indices $1\le p\le 2\le q\le \infty$ and $r>(1/p-1/q)$. 
\end{thm}

\bibliographystyle{plain}
%\bibliography{../../../Bibliography/database-NEW}

\begin{thebibliography}{10}

\bibitem{Eastham:1989}
M.~S.~P. Eastham.
\newblock {\em The asymptotic solution of linear differential systems.
  {A}pplications of the {L}evinson {T}heorem}.
\newblock Number~4 in London Mathematical Society Monographs, New Series.
  Clarendon Press, Oxford, 1989.

\bibitem{Hirosawa:2003}
F.~Hirosawa and M.~Reissig.
\newblock {From wave to {K}lein-{G}ordon type decay rates}.
\newblock In S.~Albeverio, E.~Schrohe, M.~Demuth, and {B.-W.} Schulze, editors,
  {\em Nonlinear {H}yperbolic {E}quations, {S}pectral {T}heory and {W}avelet
  {T}ransformations}, volume 145 of {\em Operator Theory, Advances and
  Applications}, pages 95--155. Birkh{\"a}user {V}erlag, Basel, 2003.

\bibitem{Hirosawa:2004a}
F.~Hirosawa and M.~Reissig.
\newblock {Well-posedness in Sobolev spaces for second-order strictly
  hyperbolic equations with nondifferentiable oscillating coefficients.}
\newblock {\em Ann. Global Anal. Geom.}, 25(2):99--119, 2004.

\bibitem{Jachmann:2008}
K.~Jachmann.
\newblock {\em A Uniform Treatment of Models of Thermoelasticity.}
\newblock PhD thesis, TU Bergakademie Freiberg, 2008.

\bibitem{Kato:1980}
T.~Kato.
\newblock {\em {Perturbation theory for linear operators. Corr. printing of the
  2nd ed.}}
\newblock {Grundlehren der mathematischen Wissenschaften, 132.
  Berlin-Heidelberg-New York: Springer-Verlag.}, 1980.

\bibitem{Knopp:1952}
K.~Knopp.
\newblock {\em {Elements of the theory of functions. Translated by F.
  Bagemihl.}}
\newblock {New York: Dover Co. 140 p. }, 1952.

\bibitem{Kubo:2003a}
A.~Kubo and M.~Reissig.
\newblock {$C^{\infty}$-well posedness of the Cauchy problem for quasi-linear
  hyperbolic equations with coefficients non-Lipschitz in time and smooth in
  space.}
\newblock {Picard, Rainer (ed.) et al., Evolution equations. Propagation
  phenomena, global existence, influence on non-linearities. Based on the
  workshop, Warsaw, Poland, July 1--July 7, 2001. Warsaw: Polish Academy of
  Sciences, Institute of Mathematics. Banach Cent. Publ. 60, 131-150}, 2003.

\bibitem{Kubo:2003}
A.~Kubo and M.~Reissig.
\newblock {Construction of parametrix to strictly hyperbolic Cauchy problems
  with fast oscillations in non-Lipschitz coefficients.}
\newblock {\em Commun. Partial Differ. Equations}, 28(7-8):1471--1502, 2003.

\bibitem{Reissig:2005}
M.~Reissig and J.~Smith.
\newblock {$L^p$-$L^q$ estimate for wave equation with bounded time dependent
  coefficient}.
\newblock {\em Hokkaido Math. J.}, 34(3):541--586, 2005.

\bibitem{Reissig:2005a}
M.~Reissig and Y.-G. Wang.
\newblock {Cauchy problems for linear thermoelastic systems of type III in one
  space variable.}
\newblock {\em Math. Methods Appl. Sci.}, 28(11):1359--1381, 2005.

\bibitem{Reissig:2006}
M.~Reissig and J.~Wirth.
\newblock {$L^p$--$L^q$} decay estimates for wave equations with monotone
  time-dependent dissipation.
\newblock In N.~Yamada, editor, {\em Mathematical Models of Phenomena and
  Evolution Equations}, K\^oky\^uroku, Nr. 1475, pages 91--106. RIMS, Kyoto
  University, 2006.

\bibitem{Wirth:2007b}
M.~Reissig and J.~Wirth.
\newblock Anisotropic thermo-elasiticity in {2D} -- {P}art {I}: {A} unified
  treatment.
\newblock {\em Asymptot. Anal.}, 57(1-2):1--27, 2008.

\bibitem{Reissig:2000}
M.~Reissig and K.~Yagdjian.
\newblock {$L\sb p$}-{$L\sb q$} decay estimates for the solutions of strictly
  hyperbolic equations of second order with increasing in time coefficients.
\newblock {\em Math. Nachr.}, 214:71--104, 2000.

\bibitem{Ruzhansky:2007}
M.~Ruzhansky and J.~Smith.
\newblock Dispersive and {S}trichartz estimates for hyperbolic equations with
  constant coefficients.
\newblock Preprint, arXiv:0711.2138.

\bibitem{Stein:1970}
E.~M. Stein.
\newblock {\em Singular integrals and differentiability properties of
  functions}.
\newblock Princeton Mathematical Series, No. 30. Princeton University Press,
  Princeton, N.J., 1970.

\bibitem{Taylor:1975}
M.~E. Taylor.
\newblock {Reflection of singularities of solutions to systems of differential
  equations.}
\newblock {\em Commun. Pure Appl. Math.}, 28:457--478, 1975.

\bibitem{Wang:2003}
Y.-G. Wang.
\newblock {A new approach to study hyperbolic-parabolic coupled systems.}
\newblock {Picard, Rainer (ed.) et al., Evolution equations. Propagation
  phenomena, global existence, influence on non-linearities. Based on the
  workshop, Warsaw, Poland, July 1-July 7, 2001. Warsaw: Polish Academy of
  Sciences, Institute of Mathematics, Banach Cent. Publ. 60, 227--236}, 2003.

\bibitem{Wang:2003a}
Y.-G. Wang.
\newblock {Microlocal analysis in nonlinear thermoelasticity.}
\newblock {\em Nonlinear Anal.}, 54:683--705, 2003.

\bibitem{Wirth:2006}
J.~Wirth.
\newblock {Wave equations with time-dependent dissipation. I: Non-effective
  dissipation}.
\newblock {\em J. Differ. Equations}, 222(2):487--514, 2006.

\bibitem{Wirth:2007}
J.~Wirth.
\newblock {Wave equations with time-dependent dissipation. II: Effective
  dissipation}.
\newblock {\em J. Differ. Equations}, 232(1):74--103, 2007.

\bibitem{Wirth:2007c}
J.~Wirth.
\newblock Anisotropic thermo-elasiticity in {2D} -- {P}art {II}:
  {A}pplications.
\newblock {\em Asymptotic Anal.}, 57(1-2):29--40, 2008.

\bibitem{Wirth:2007d}
J.~Wirth.
\newblock On the influence of time-periodic dissipation on energy and
  dispersive estimates.
\newblock {\em Hiroshima Math. J.}, 38(3), 2008.
\newblock to appear.

\bibitem{Yagdjian:1997}
K.~Yagdjian.
\newblock {\em The {C}auchy problem for hyperbolic operators}, volume~12 of
  {\em Mathematical Topics}.
\newblock Akademie Verlag, Berlin, 1997.

\end{thebibliography}

\end{document}